\documentclass[11pt,a4paper]{article}
\RequirePackage{amsmath}
\RequirePackage{amssymb,latexsym}
\RequirePackage{amsthm}
\RequirePackage{enumerate}
\RequirePackage[hang,flushmargin,bottom,stable]{footmisc}
\RequirePackage[margin=0.35in,format=hang,font=small,labelfont=bf]{caption}
\RequirePackage[normalem]{ulem}
\RequirePackage{needspace}
\RequirePackage{graphicx,tikz}
\RequirePackage[T1]{fontenc}
\RequirePackage[sc]{mathpazo}
\RequirePackage[colorlinks=true,urlcolor=blue,linkcolor=blue,citecolor=blue]{hyperref}
\usepackage{multirow}
\usepackage{setspace}

\hypersetup{
pdfstartview={XYZ null null 1.00},
pdflang={en-GB},
}

\newtheorem{thmO}{Theorem}
\newtheorem{propO}[thmO]{Proposition}

\newenvironment{bulletnums} {\vspace{-9pt}\begin{enumerate}\itemsep0pt} {\end{enumerate}\vspace{-9pt}}

\newcommand{\ba}{\boldsymbol{a}}
\newcommand{\eq}{\mathchoice{\;=\;}{=}{=}{=}}
\newcommand{\Erdos}{Erd\H{o}s}
\newcommand{\geqs}{\geqslant}
\newcommand{\leqs}{\leqslant}
\newcommand{\limsupinfty}[1][n]{\limsup\limits_{#1\rightarrow\infty}}
\newcommand{\msf}[1]{\text{\textsf{\textup{#1}}}}
\newcommand{\msfi}[1]{\text{\textsf{#1}}}
\newcommand{\pho}{\phantom{0}}
\newcommand{\sleqs}{\mathchoice{\;\leqs\;}{\leqs}{\leqs}{\leqs}}

\newcommand{\myTitle}{On balancing two-slice portions of cake}
\title{\textbf{\myTitle}}

\author{$\phantom{{}^\dagger}$David Bevan${}^\dagger$}

\hypersetup{
pdftitle={\myTitle},
pdfauthor={David Bevan},
pdfstartview={XYZ null null 1.00}
}

\date{}

\begin{document}
\maketitle

{\begin{NoHyper}
\let\thefootnote\relax\footnotetext
{${}^\dagger$The University of Strathclyde, Glasgow, Scotland; email: dibevan@hotmail.co.uk.}
\end{NoHyper}}

{\begin{NoHyper}
\let\thefootnote\relax
\end{NoHyper}}

\begin{abstract}
\noindent
After $n$ radial cuts of a circular cake, it is divided into $n$ slices.
Call an adjacent pair of slices a \emph{portion}.
We exhibit an infinite sequence of cuts such that the ratio between the maximum and minimum sizes of a portion never exceeds~1.755.
This improves on the trivial upper bound of~2, disproving a conjecture of Korsky.
\end{abstract}

\section{Introduction}

De~Bruijn and \Erdos{}~\cite{DeBE1949} were the first to investigate the sequential interval discrepancy problem on a circle.
Given an infinite sequence $\ba=(a_i)_{i=1}^\infty$ of points on a circle --- or equivalently an infinite sequence of radial cuts of a circular cake --- the first $n$ cuts split the cake into $n$ slices.
Let $\mu^r_n(\ba)$ denote the ratio of the maximum size of $r$ consecutive slices to the minimum size of $r$ consecutive slices.
De~Bruijn and \Erdos{} 
were interested in the value of 
\[
\mu_r \eq \inf_{\ba} \,\, \limsupinfty \mu^r_n(\ba) ,
\]
the smallest possible asymptotic least upper bound of the ratios.
They established that $\mu_1=2$ and proved that $\mu_r\geqs1+1/r$ for every~$r$.
They also conjectured that $r(\mu_r-1)$ tends to infinity with increasing~$r$.

Recent work has seen some progress on these questions.
On the one hand, Korsky~\cite{Korsky2026} improved the lower bound to $\mu_r\geqs1+r/(r^2-1)$.
He also conjectured that~$\mu_2=2$.
On the other hand, Cl\'ement and Steinerberger~\cite{CS2025} established that $r(\mu_r-1)\leqs c\log r$ for some absolute constant~$c$.
Also, in another recent paper~\cite{BevanCakeSlicing}, we constructed sequences of cuts from which upper bounds on $\mu_r$ for small values of~$r$ could be calculated:
\[
\mu_3 \sleqs 3/2 , \qquad
\mu_4 \sleqs 3/2 , \qquad
\mu_5 \sleqs 7/5 , \qquad
\mu_6 \sleqs 4/3 , \qquad
\mu_7 \sleqs 5/4 .
\]

Our goal in this short note is to improve on the trivial upper bound of 2 for $\mu_2$. 

\section{A nontrivial upper bound on \texorpdfstring{$\mu_2$}{mu2}}

Korsky proved that $\mu_2\geqs5/3$. He also conjectured 
that $\mu_2=2$.
We disprove Korsky's conjecture.

\begin{thmO}\label{thm}
  If $\rho\approx0.75488$ is the real root of $\rho^2+\rho^3=1$, then $\mu_2\leqs1+\rho$.
\end{thmO}

\begin{proof}
Call an adjacent pair of slices a \emph{portion}.
The following process defines an infinite sequence of cuts such that the ratio between the maximum
and minimum sizes of a portion never exceeds $1+\rho$:

After an initial cut (so that the cake consists of one slice), we repeatedly chose a largest slice (which may be chosen arbitrarily), of size~$s$ say, and cut it to create two slices with sizes $\rho^2s\approx0.56984s$ and $\rho^3s\approx0.43016s$, placing the larger of the two new slices before the smaller unless this would result in adjacent slices having the same size.
We establish below in Proposition~\ref{propNoRepeat} that this ensures that at no time do two adjacent slices have the same size.

The first few cuts are illustrated (linearly) in Figure~\ref{fig}.
Here the leftmost largest slice is always chosen to be cut.

\begin{figure}[ht]
\centering
\begin{tikzpicture}[scale=.2]
\footnotesize
\newcommand{\xWid}{72}
\newcommand{\yHeight}{2.5}
\newcommand{\yInc}{-3.25}
\node at (-2,1*\yInc+\yHeight/2) {\pho1:};
\draw[fill=yellow!6] (0*\xWid,1*\yInc) rectangle (1.*\xWid,1*\yInc+\yHeight);
\node at (0.5*\xWid,1*\yInc+\yHeight/2) {0};

\node at (-2,2*\yInc+\yHeight/2) {\pho2:};
\draw[fill=blue!18] (0*\xWid,2*\yInc) rectangle (0.56984*\xWid,2*\yInc+\yHeight);
\node at (0.28492*\xWid,2*\yInc+\yHeight/2) {2};
\draw[fill=blue!24] (0.56984*\xWid,2*\yInc) rectangle (1.*\xWid,2*\yInc+\yHeight);
\node at (0.78492*\xWid,2*\yInc+\yHeight/2) {3};

\node at (-2,3*\yInc+\yHeight/2) {\pho3:};
\draw[fill=blue!30] (0*\xWid,3*\yInc) rectangle (0.324718*\xWid,3*\yInc+\yHeight);
\node at (0.162359*\xWid,3*\yInc+\yHeight/2) {4};
\draw[fill=blue!36] (0.324718*\xWid,3*\yInc) rectangle (0.56984*\xWid,3*\yInc+\yHeight);
\node at (0.447279*\xWid,3*\yInc+\yHeight/2) {5};
\draw[fill=yellow!24] (0.56984*\xWid,3*\yInc) rectangle (1.*\xWid,3*\yInc+\yHeight);
\node at (0.78492*\xWid,3*\yInc+\yHeight/2) {3};

\node at (-2,4*\yInc+\yHeight/2) {\pho4:};
\draw[fill=yellow!30] (0*\xWid,4*\yInc) rectangle (0.324718*\xWid,4*\yInc+\yHeight);
\node at (0.162359*\xWid,4*\yInc+\yHeight/2) {4};
\draw[fill=yellow!36] (0.324718*\xWid,4*\yInc) rectangle (0.56984*\xWid,4*\yInc+\yHeight);
\node at (0.447279*\xWid,4*\yInc+\yHeight/2) {5};
\draw[fill=blue!42] (0.56984*\xWid,4*\yInc) rectangle (0.754878*\xWid,4*\yInc+\yHeight);
\node at (0.662359*\xWid,4*\yInc+\yHeight/2) {6};
\draw[fill=blue!36] (0.754878*\xWid,4*\yInc) rectangle (1.*\xWid,4*\yInc+\yHeight);
\node at (0.877439*\xWid,4*\yInc+\yHeight/2) {5};

\node at (-2,5*\yInc+\yHeight/2) {\pho5:};
\draw[fill=blue!42] (0*\xWid,5*\yInc) rectangle (0.185037*\xWid,5*\yInc+\yHeight);
\node at (0.0925187*\xWid,5*\yInc+\yHeight/2) {6};
\draw[fill=blue!48] (0.185037*\xWid,5*\yInc) rectangle (0.324718*\xWid,5*\yInc+\yHeight);
\node at (0.254878*\xWid,5*\yInc+\yHeight/2) {7};
\draw[fill=yellow!36] (0.324718*\xWid,5*\yInc) rectangle (0.56984*\xWid,5*\yInc+\yHeight);
\node at (0.447279*\xWid,5*\yInc+\yHeight/2) {5};
\draw[fill=yellow!42] (0.56984*\xWid,5*\yInc) rectangle (0.754878*\xWid,5*\yInc+\yHeight);
\node at (0.662359*\xWid,5*\yInc+\yHeight/2) {6};
\draw[fill=yellow!36] (0.754878*\xWid,5*\yInc) rectangle (1.*\xWid,5*\yInc+\yHeight);
\node at (0.877439*\xWid,5*\yInc+\yHeight/2) {5};

\node at (-2,6*\yInc+\yHeight/2) {\pho6:};
\draw[fill=yellow!42] (0*\xWid,6*\yInc) rectangle (0.185037*\xWid,6*\yInc+\yHeight);
\node at (0.0925187*\xWid,6*\yInc+\yHeight/2) {6};
\draw[fill=yellow!48] (0.185037*\xWid,6*\yInc) rectangle (0.324718*\xWid,6*\yInc+\yHeight);
\node at (0.254878*\xWid,6*\yInc+\yHeight/2) {7};
\draw[fill=blue!54] (0.324718*\xWid,6*\yInc) rectangle (0.43016*\xWid,6*\yInc+\yHeight);
\node at (0.377439*\xWid,6*\yInc+\yHeight/2) {8};
\draw[fill=blue!48] (0.43016*\xWid,6*\yInc) rectangle (0.56984*\xWid,6*\yInc+\yHeight);
\node at (0.5*\xWid,6*\yInc+\yHeight/2) {7};
\draw[fill=yellow!42] (0.56984*\xWid,6*\yInc) rectangle (0.754878*\xWid,6*\yInc+\yHeight);
\node at (0.662359*\xWid,6*\yInc+\yHeight/2) {6};
\draw[fill=yellow!36] (0.754878*\xWid,6*\yInc) rectangle (1.*\xWid,6*\yInc+\yHeight);
\node at (0.877439*\xWid,6*\yInc+\yHeight/2) {5};

\node at (-2,7*\yInc+\yHeight/2) {\pho7:};
\draw[fill=yellow!42] (0*\xWid,7*\yInc) rectangle (0.185037*\xWid,7*\yInc+\yHeight);
\node at (0.0925187*\xWid,7*\yInc+\yHeight/2) {6};
\draw[fill=yellow!48] (0.185037*\xWid,7*\yInc) rectangle (0.324718*\xWid,7*\yInc+\yHeight);
\node at (0.254878*\xWid,7*\yInc+\yHeight/2) {7};
\draw[fill=yellow!54] (0.324718*\xWid,7*\yInc) rectangle (0.43016*\xWid,7*\yInc+\yHeight);
\node at (0.377439*\xWid,7*\yInc+\yHeight/2) {8};
\draw[fill=yellow!48] (0.43016*\xWid,7*\yInc) rectangle (0.56984*\xWid,7*\yInc+\yHeight);
\node at (0.5*\xWid,7*\yInc+\yHeight/2) {7};
\draw[fill=yellow!42] (0.56984*\xWid,7*\yInc) rectangle (0.754878*\xWid,7*\yInc+\yHeight);
\node at (0.662359*\xWid,7*\yInc+\yHeight/2) {6};
\draw[fill=blue!48] (0.754878*\xWid,7*\yInc) rectangle (0.894558*\xWid,7*\yInc+\yHeight);
\node at (0.824718*\xWid,7*\yInc+\yHeight/2) {7};
\draw[fill=blue!54] (0.894558*\xWid,7*\yInc) rectangle (1.*\xWid,7*\yInc+\yHeight);
\node at (0.947279*\xWid,7*\yInc+\yHeight/2) {8};

\node at (-2,8*\yInc+\yHeight/2) {\pho8:};
\draw[fill=blue!60] (0*\xWid,8*\yInc) rectangle (0.0795956*\xWid,8*\yInc+\yHeight);
\node at (0.0397978*\xWid,8*\yInc+\yHeight/2) {9};
\draw[fill=blue!54] (0.0795956*\xWid,8*\yInc) rectangle (0.185037*\xWid,8*\yInc+\yHeight);
\node at (0.132316*\xWid,8*\yInc+\yHeight/2) {8};
\draw[fill=yellow!48] (0.185037*\xWid,8*\yInc) rectangle (0.324718*\xWid,8*\yInc+\yHeight);
\node at (0.254878*\xWid,8*\yInc+\yHeight/2) {7};
\draw[fill=yellow!54] (0.324718*\xWid,8*\yInc) rectangle (0.43016*\xWid,8*\yInc+\yHeight);
\node at (0.377439*\xWid,8*\yInc+\yHeight/2) {8};
\draw[fill=yellow!48] (0.43016*\xWid,8*\yInc) rectangle (0.56984*\xWid,8*\yInc+\yHeight);
\node at (0.5*\xWid,8*\yInc+\yHeight/2) {7};
\draw[fill=yellow!42] (0.56984*\xWid,8*\yInc) rectangle (0.754878*\xWid,8*\yInc+\yHeight);
\node at (0.662359*\xWid,8*\yInc+\yHeight/2) {6};
\draw[fill=yellow!48] (0.754878*\xWid,8*\yInc) rectangle (0.894558*\xWid,8*\yInc+\yHeight);
\node at (0.824718*\xWid,8*\yInc+\yHeight/2) {7};
\draw[fill=yellow!54] (0.894558*\xWid,8*\yInc) rectangle (1.*\xWid,8*\yInc+\yHeight);
\node at (0.947279*\xWid,8*\yInc+\yHeight/2) {8};

\node at (-2,9*\yInc+\yHeight/2) {\pho9:};
\draw[fill=yellow!60] (0*\xWid,9*\yInc) rectangle (0.0795956*\xWid,9*\yInc+\yHeight);
\node at (0.0397978*\xWid,9*\yInc+\yHeight/2) {9};
\draw[fill=yellow!54] (0.0795956*\xWid,9*\yInc) rectangle (0.185037*\xWid,9*\yInc+\yHeight);
\node at (0.132316*\xWid,9*\yInc+\yHeight/2) {8};
\draw[fill=yellow!48] (0.185037*\xWid,9*\yInc) rectangle (0.324718*\xWid,9*\yInc+\yHeight);
\node at (0.254878*\xWid,9*\yInc+\yHeight/2) {7};
\draw[fill=yellow!54] (0.324718*\xWid,9*\yInc) rectangle (0.43016*\xWid,9*\yInc+\yHeight);
\node at (0.377439*\xWid,9*\yInc+\yHeight/2) {8};
\draw[fill=yellow!48] (0.43016*\xWid,9*\yInc) rectangle (0.56984*\xWid,9*\yInc+\yHeight);
\node at (0.5*\xWid,9*\yInc+\yHeight/2) {7};
\draw[fill=blue!54] (0.56984*\xWid,9*\yInc) rectangle (0.675282*\xWid,9*\yInc+\yHeight);
\node at (0.622561*\xWid,9*\yInc+\yHeight/2) {8};
\draw[fill=blue!60] (0.675282*\xWid,9*\yInc) rectangle (0.754878*\xWid,9*\yInc+\yHeight);
\node at (0.71508*\xWid,9*\yInc+\yHeight/2) {9};
\draw[fill=yellow!48] (0.754878*\xWid,9*\yInc) rectangle (0.894558*\xWid,9*\yInc+\yHeight);
\node at (0.824718*\xWid,9*\yInc+\yHeight/2) {7};
\draw[fill=yellow!54] (0.894558*\xWid,9*\yInc) rectangle (1.*\xWid,9*\yInc+\yHeight);
\node at (0.947279*\xWid,9*\yInc+\yHeight/2) {8};

\node at (-2,10*\yInc+\yHeight/2) {10:};
\draw[fill=yellow!60] (0*\xWid,10*\yInc) rectangle (0.0795956*\xWid,10*\yInc+\yHeight);
\node at (0.0397978*\xWid,10*\yInc+\yHeight/2) {9};
\draw[fill=yellow!54] (0.0795956*\xWid,10*\yInc) rectangle (0.185037*\xWid,10*\yInc+\yHeight);
\node at (0.132316*\xWid,10*\yInc+\yHeight/2) {8};
\draw[fill=blue!60] (0.185037*\xWid,10*\yInc) rectangle (0.264633*\xWid,10*\yInc+\yHeight);
\node at (0.224835*\xWid,10*\yInc+\yHeight/2) {9};
\draw[fill=blue!66] (0.264633*\xWid,10*\yInc) rectangle (0.324718*\xWid,10*\yInc+\yHeight);
\node at (0.294675*\xWid,10*\yInc+\yHeight/2) {10};
\draw[fill=yellow!54] (0.324718*\xWid,10*\yInc) rectangle (0.43016*\xWid,10*\yInc+\yHeight);
\node at (0.377439*\xWid,10*\yInc+\yHeight/2) {8};
\draw[fill=yellow!48] (0.43016*\xWid,10*\yInc) rectangle (0.56984*\xWid,10*\yInc+\yHeight);
\node at (0.5*\xWid,10*\yInc+\yHeight/2) {7};
\draw[fill=yellow!54] (0.56984*\xWid,10*\yInc) rectangle (0.675282*\xWid,10*\yInc+\yHeight);
\node at (0.622561*\xWid,10*\yInc+\yHeight/2) {8};
\draw[fill=yellow!60] (0.675282*\xWid,10*\yInc) rectangle (0.754878*\xWid,10*\yInc+\yHeight);
\node at (0.71508*\xWid,10*\yInc+\yHeight/2) {9};
\draw[fill=yellow!48] (0.754878*\xWid,10*\yInc) rectangle (0.894558*\xWid,10*\yInc+\yHeight);
\node at (0.824718*\xWid,10*\yInc+\yHeight/2) {7};
\draw[fill=yellow!54] (0.894558*\xWid,10*\yInc) rectangle (1.*\xWid,10*\yInc+\yHeight);
\node at (0.947279*\xWid,10*\yInc+\yHeight/2) {8};

\node at (-2,11*\yInc+\yHeight/2) {11:};
\draw[fill=yellow!60] (0*\xWid,11*\yInc) rectangle (0.0795956*\xWid,11*\yInc+\yHeight);
\node at (0.0397978*\xWid,11*\yInc+\yHeight/2) {9};
\draw[fill=yellow!54] (0.0795956*\xWid,11*\yInc) rectangle (0.185037*\xWid,11*\yInc+\yHeight);
\node at (0.132316*\xWid,11*\yInc+\yHeight/2) {8};
\draw[fill=yellow!60] (0.185037*\xWid,11*\yInc) rectangle (0.264633*\xWid,11*\yInc+\yHeight);
\node at (0.224835*\xWid,11*\yInc+\yHeight/2) {9};
\draw[fill=yellow!66] (0.264633*\xWid,11*\yInc) rectangle (0.324718*\xWid,11*\yInc+\yHeight);
\node at (0.294675*\xWid,11*\yInc+\yHeight/2) {10};
\draw[fill=yellow!54] (0.324718*\xWid,11*\yInc) rectangle (0.43016*\xWid,11*\yInc+\yHeight);
\node at (0.377439*\xWid,11*\yInc+\yHeight/2) {8};
\draw[fill=blue!60] (0.43016*\xWid,11*\yInc) rectangle (0.509755*\xWid,11*\yInc+\yHeight);
\node at (0.469958*\xWid,11*\yInc+\yHeight/2) {9};
\draw[fill=blue!66] (0.509755*\xWid,11*\yInc) rectangle (0.56984*\xWid,11*\yInc+\yHeight);
\node at (0.539798*\xWid,11*\yInc+\yHeight/2) {10};
\draw[fill=yellow!54] (0.56984*\xWid,11*\yInc) rectangle (0.675282*\xWid,11*\yInc+\yHeight);
\node at (0.622561*\xWid,11*\yInc+\yHeight/2) {8};
\draw[fill=yellow!60] (0.675282*\xWid,11*\yInc) rectangle (0.754878*\xWid,11*\yInc+\yHeight);
\node at (0.71508*\xWid,11*\yInc+\yHeight/2) {9};
\draw[fill=yellow!48] (0.754878*\xWid,11*\yInc) rectangle (0.894558*\xWid,11*\yInc+\yHeight);
\node at (0.824718*\xWid,11*\yInc+\yHeight/2) {7};
\draw[fill=yellow!54] (0.894558*\xWid,11*\yInc) rectangle (1.*\xWid,11*\yInc+\yHeight);
\node at (0.947279*\xWid,11*\yInc+\yHeight/2) {8};

\node at (-2,12*\yInc+\yHeight/2) {12:};
\draw[fill=yellow!60] (0*\xWid,12*\yInc) rectangle (0.0795956*\xWid,12*\yInc+\yHeight);
\node at (0.0397978*\xWid,12*\yInc+\yHeight/2) {9};
\draw[fill=yellow!54] (0.0795956*\xWid,12*\yInc) rectangle (0.185037*\xWid,12*\yInc+\yHeight);
\node at (0.132316*\xWid,12*\yInc+\yHeight/2) {8};
\draw[fill=yellow!60] (0.185037*\xWid,12*\yInc) rectangle (0.264633*\xWid,12*\yInc+\yHeight);
\node at (0.224835*\xWid,12*\yInc+\yHeight/2) {9};
\draw[fill=yellow!66] (0.264633*\xWid,12*\yInc) rectangle (0.324718*\xWid,12*\yInc+\yHeight);
\node at (0.294675*\xWid,12*\yInc+\yHeight/2) {10};
\draw[fill=yellow!54] (0.324718*\xWid,12*\yInc) rectangle (0.43016*\xWid,12*\yInc+\yHeight);
\node at (0.377439*\xWid,12*\yInc+\yHeight/2) {8};
\draw[fill=yellow!60] (0.43016*\xWid,12*\yInc) rectangle (0.509755*\xWid,12*\yInc+\yHeight);
\node at (0.469958*\xWid,12*\yInc+\yHeight/2) {9};
\draw[fill=yellow!66] (0.509755*\xWid,12*\yInc) rectangle (0.56984*\xWid,12*\yInc+\yHeight);
\node at (0.539798*\xWid,12*\yInc+\yHeight/2) {10};
\draw[fill=yellow!54] (0.56984*\xWid,12*\yInc) rectangle (0.675282*\xWid,12*\yInc+\yHeight);
\node at (0.622561*\xWid,12*\yInc+\yHeight/2) {8};
\draw[fill=yellow!60] (0.675282*\xWid,12*\yInc) rectangle (0.754878*\xWid,12*\yInc+\yHeight);
\node at (0.71508*\xWid,12*\yInc+\yHeight/2) {9};
\draw[fill=blue!66] (0.754878*\xWid,12*\yInc) rectangle (0.814963*\xWid,12*\yInc+\yHeight);
\node at (0.78492*\xWid,12*\yInc+\yHeight/2) {10};
\draw[fill=blue!60] (0.814963*\xWid,12*\yInc) rectangle (0.894558*\xWid,12*\yInc+\yHeight);
\node at (0.85476*\xWid,12*\yInc+\yHeight/2) {9};
\draw[fill=yellow!54] (0.894558*\xWid,12*\yInc) rectangle (1.*\xWid,12*\yInc+\yHeight);
\node at (0.947279*\xWid,12*\yInc+\yHeight/2) {8};
\end{tikzpicture}
\caption{The first twelve cuts: new slices are shown in blue; slices of size $\rho^k$ are labelled $k$.}\label{fig}
\end{figure}
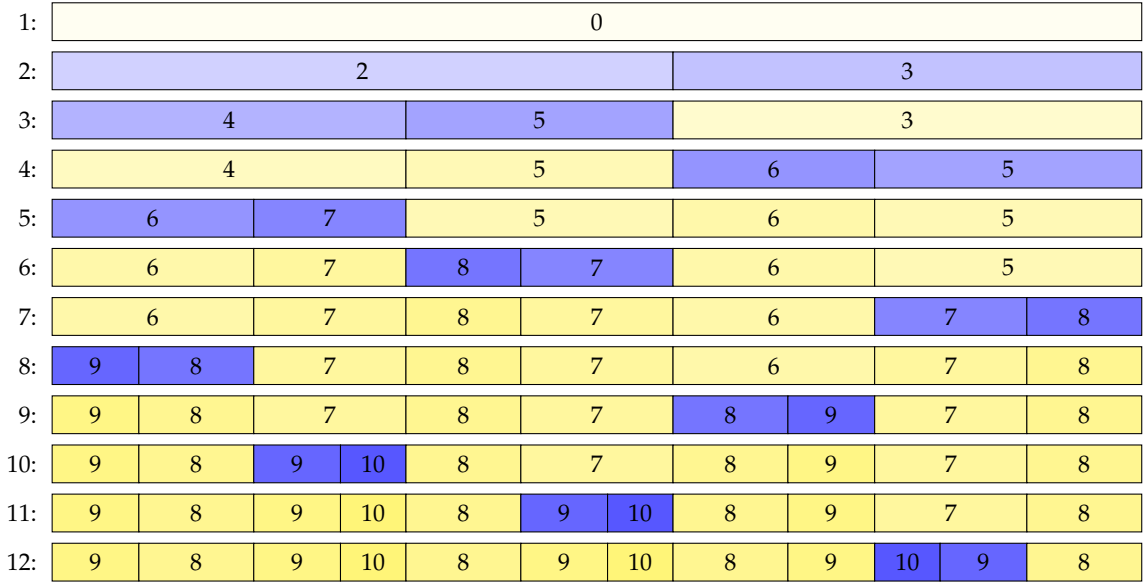

After any number of cuts, slices can take at most four distinct sizes, in the ratios \mbox{$1:\rho:\rho^2:\rho^3$}.
On the assumption that no two adjacent slices have the same size, the greatest possible ratio between the maximum and minimum sizes of a two-slice portion is thus
\[
\frac{1+\rho}{\rho^2+\rho^3} \eq 1+\rho . \qedhere
\]
\end{proof}

It remains to guarantee that we can avoid adjacent slices of the same size.
To do so, by re-scaling we represent a sliced cake by a word over the alphabet $\{\msf0,\msf1,\msf2,\msf3\}$.
For example, the penultimate cake in Figure~\ref{fig} (with eleven slices) is represented by \msf{21231231201}.

It then follows from 
the following proposition ---
by induction on the number of cuts --- that we can avoid creating adjacent slices of the same size, since the requirements of the proposition are satisfied after the first three cuts.

\begin{propO}\label{propNoRepeat}
  Suppose $w$ is a finite word over the alphabet $\{\msf0,\msf1,\msf2,\msf3\}$ of length at least three, with no cyclic factor in 
  \[
  A\eq\{\, \msf{00},\msf{11},\msf{22},\msf{33}, \; \msf{102},\msf{103},\msf{202},\msf{203},\msf{301},\msf{302},\msf{303}, \; \msf{013},\msf{213},\msf{313} \,\} .
  \]
  If $w$ contains at least one~\msf0, then    
  let $w_{\msf{23}}$ be the word formed from $w$ by replacing some \msf0 with~\msf{23}, and
  $w_{\msf{32}}$ the word formed from $w$ by replacing the same \msf0 with~\msf{32}.
  Then either $w_{\msf{23}}$ or $w_{\msf{32}}$ cyclically avoids all the factors in~$A$.
  
  Moreover, if $w$ contains no~\msf0, then the word $w^-$ formed from $w$ by replacing each letter $\msfi{k}$ with $\msfi{k}-1$ also cyclically avoids every factor in~$A$.
\end{propO}

\begin{proof}
  To confirm avoidance of factors in $A$, we need only consider the two letters immediately before and the two letters immediately after a \msf0 in $w$.
  Since $w$ has length at least three, these letters are unchanged in $w_{\msf{23}}$ and $w_{\msf{32}}$.
  
  Since $w$ avoids every forbidden factor in~$A$, any \msf0 must be preceded by either \msf1 or \msf2 and must be followed by~\msf1.
  Moreover, the \msf1 following the \msf0 must itself be followed by either \msf0 or~\msf2.
  There are thus just two cases:
  
\begin{bulletnums}
  \item If the \msf0 occurs in a \msf{101}, then $w_{\msf{23}}$ avoids all the forbidden factors:
  \[
  \{\msf0,\msf2,\msf3\}\msf{101}\{\msf0,\msf2\} \;\mapsto\; \{\msf0,\msf2,\msf3\}\msf{1231}\{\msf0,\msf2\} .
  \]
  \item If the \msf0 occurs in a \msf{201}, then $w_{\msf{32}}$ avoids all the forbidden factors:
  \[
  \{\msf0,\msf1,\msf3\}\msf{201}\{\msf0,\msf2\} \;\mapsto\; \{\msf0,\msf1,\msf3\}\msf{2321}\{\msf0,\msf2\} .
  \]
\end{bulletnums}  
  
If $w$ contains no~\msf0, then 
$w^-$
only contains \msf0, \msf1 and \msf2. 
Now, $w^-$ can't have \msf{00}, \msf{11}, \msf{22}, \msf{102} or \msf{202} as a factor since $w$ avoids \msf{11}, \msf{22}, \msf{33}, \msf{213} and \msf{313}. 
Hence, $w^-$ also avoids every factor in~$A$.
\end{proof}


\vspace{12pt}
\begin{flushright}
\emph{Soli Deo gloria!}
\end{flushright}

\bibliographystyle{plain}
{\footnotesize\bibliography{../bib/mybib}}

\end{document}